\documentclass[11pt, reqno]{amsart}

\usepackage[top=4.5cm, bottom=3.5cm, left=3.5cm, right=3.5cm]{geometry}
\usepackage{amsmath}
\usepackage{amsthm}
\usepackage{amsfonts}
\usepackage{amssymb}
\usepackage{eucal}
\usepackage{bm}
\usepackage[all,cmtip]{xy}
\usepackage{hyperref}
\usepackage{color}
\usepackage{mathtools}
\usepackage{tabu}

\numberwithin{equation}{section}

\theoremstyle{plain}
\newtheorem{theorem}{Theorem}[section]
\newtheorem{lemma}[theorem]{Lemma}
\newtheorem{proposition}[theorem]{Proposition}
\newtheorem{corollary}[theorem]{Corollary}

\theoremstyle{definition}
\newtheorem{definition}[theorem]{Definition}
\newtheorem{remark}[theorem]{Remark}

\newtheorem{example}[theorem]{Example}

\DeclareMathOperator{\colim}{\mathrm{colim}}
\DeclareMathOperator{\QCoh}{QCoh}
\DeclareMathOperator{\Perf}{Perf}
\DeclareMathOperator{\Vect}{Vect}
\newcommand{\Vectfd}[1]{\mathrm{Vect_{#1}^{fd}}}
\DeclareMathOperator{\Map}{Map}
\DeclareMathOperator{\ulMap}{\underline{Map}}
\newcommand{\Fun}{\mathrm{Fun}}
\newcommand{\StCat}{\mathrm{Cat}_{\bC}}
\newcommand{\PrCat}{\mathrm{PrCat}_{\bC}}
\newcommand{\Res}{\mathrm{Res}}
\newcommand{\Ind}{\mathrm{Ind}}
\newcommand{\Nm}{\mathrm{Nm}}

\newcommand{\Av}{\mathrm{Av}}

\newcommand{\op}{\mathrm{op}}
\newcommand{\opp}{\mathrm{op}}

\DeclareMathOperator{\Spec}{Spec}

\newcommand{\wtilde}{\widetilde}
\newcommand{\characteristic}{\mathrm{char}}
\newcommand{\id}{\mathrm{id}}

\DeclareMathOperator{\HH}{HH}

\DeclareMathOperator{\Ext}{Ext}
\DeclareMathOperator{\HC}{HC}

\newcommand{\cO}{\mathcal{O}}

\newcommand{\cC}{\mathcal{C}}
\newcommand{\calD}{\mathcal{D}}
\newcommand{\cE}{\mathcal{E}}

\newcommand{\cR}{\mathcal{R}}

\newcommand{\rH}{\mathrm{H}}

\newcommand{\rS}{\mathrm{S}}

\newcommand{\rN}{\mathrm{N}}

\newcommand{\rT}{\mathrm{T}}
\newcommand{\rh}{\mathrm{h}}

\newcommand{\bC}{\mathbf{k}}

\newcommand{\bZ}{\mathbf{Z}}
\newcommand{\bP}{\mathbf{P}}

\newcommand{\sO}{\mathsf{O}}

\begin{document}

\title{Hochschild cohomology and group actions}

\author{Alexander Perry}
\address{Department of Mathematics, Columbia University, New York, NY 10027, USA \smallskip}
\email{aperry@math.columbia.edu}

\thanks{This work was partially supported by an NSF postdoctoral fellowship, DMS-1606460.}

\begin{abstract}
Given a finite group action on a (suitably enhanced) triangulated category linear over a field, 
we establish a formula for the Hochschild cohomology of the category of invariants, 
assuming the order of the group is coprime to the characteristic of the base field. 
The formula shows that the cohomology splits canonically 
with one summand given by  
the invariant subspace of the Hochschild cohomology of the original category. 
We also prove that Serre functors act trivially on Hochschild cohomology, and 
combine this with our formula to give a useful mechanism for computing the 
Hochschild cohomology of fractional Calabi--Yau categories. 
\end{abstract}

\maketitle

%New section
\section{Introduction}

Let $\bC$ be an algebraically closed field. 
The Hochschild cohomology of a scheme $X$ over $\bC$ can be 
defined as the graded $\bC$-algebra 
\begin{equation*}
\HH^\bullet(X) = \Ext^\bullet_{X \times X}(\cO_{\Delta}, \cO_{\Delta}) , 
\end{equation*}
where $\Delta \subset X \times X$ denotes the diagonal. 
By the Hochschild--Kostant--Rosenberg theorem (in the form of \cite[Theorem 4.8]{yekutieli}), this invariant can 
be computed in terms of polyvector fields: 
if $X$ is a smooth and separated over $\bC$ and $\dim(X)!$ is invertible in $\bC$, 
then there is an isomorphism 
\begin{equation}
\label{HKR-HC}
\HH^n(X) \cong \bigoplus_{p+q = n} \rH^q(X, \wedge^p \rT_X). 
\end{equation}

The Hochschild cohomology of $X$ is a derived invariant, i.e. 
only depends on the category of perfect complexes $\Perf(X)$. 
Indeed, $\cO_{\Delta}$ is the Fourier--Mukai kernel for the identity functor 
of $\Perf(X)$, and $\HH^\bullet(X)$ is identified with the space of 
``derived endomorphisms'' of this functor. 
In order for this recipe for $\HH^\bullet(X)$ to make sense, we need to 
regard $\Perf(X)$ as a suitably enhanced triangulated category, e.g. as a 
pretriangulated DG category over $\bC$ or as a $\bC$-linear stable $\infty$-category. 
In this paper, we call such an enhanced category a 
\emph{$\bC$-linear category} (see \S\ref{section-stable-cats} for details). 
The upshot is that now given any $\bC$-linear category $\cC$, the 
same prescription as above defines a graded $\bC$-algebra $\HH^\bullet(\cC)$ 
known as its Hochschild cohomology. 
This is a fundamental and well-studied invariant of the category $\cC$, 
which in particular controls its infinitesimal deformation theory.  

The purpose of this paper is to study the Hochschild cohomology 
of $\cC$ in the presence of the action of a finite group $G$.  
In this situation, we can form a category $\cC^G$ of invariants. 
To give a feeling for this construction: if $\cC = \Perf(X)$ for a smooth variety $X$, 
the $G$-action on $\cC$ is induced by one on $X$, and the order of $G$ is invertible in $\bC$, 
then $\cC^G \simeq \Perf([X/G])$ where $[X/G]$ denotes the quotient stack 
(see~\cite{elagin-equivariant, sosna}). 
We address the question: 
What is the relation between $\HH^\bullet(\cC^G)$ 
and the space $\HH^\bullet(\cC)^G$ of invariants for the induced $G$-action 
on $\HH^\bullet(\cC)$?   

\begin{theorem}
\label{theorem-intro}
Let $\cC$ be a $\bC$-linear with an action by a finite group $G$. 
Assume the order of $G$ is coprime to the characteristic of $\bC$. 
Then there is a canonical splitting of 
$\HH^\bullet(\cC^G)$ 
with $\HH^\bullet(\cC)^G$ as a summand. 
\end{theorem}

In fact, we explicitly identify the complementary summand to $\HH^\bullet(\cC)^G$ 
in terms of the Hochschild cohomology of $\cC$ ``with coefficients'' in the autoequivalences 
corresponding to the elements $1 \neq g \in G$, see Theorem~\ref{theorem-equivariant-HC}. 
In the geometric situation mentioned above where $G$ acts on a smooth variety $X$, 
this recovers an orbifold HKR decomposition from \cite{orbiHKR-caldararu} 
which expresses $\HH^\bullet([X/G])$ in terms of coherent cohomology groups on the fixed loci of the elements $g \in G$, see 
Corollary~\ref{corollary-orbifold-HKR}. 

Our main application of Theorem~\ref{theorem-intro} is to studying 
the Hochschild cohomology of fractional Calabi--Yau categories 
which are not necessarily of the form $\cC = \Perf(X)$. 
Recall that $\cC$ is called \emph{fractional Calabi--Yau} if it is proper and has a Serre functor $\rS_{\cC}$ 
satisfying $\rS_{\cC}^{q} \cong [p]$ for some integers $p$ and $q \neq 0$, 
and \emph{Calabi--Yau} if we can take $q = 1$. 
The primordial example is $\cC = \Perf(X)$ 
where $X$ is a smooth proper $\bC$-scheme with torsion canonical bundle $\omega_X$.  
In \cite{kuznetsov2015calabi} Kuznetsov 
showed there is also a large, interesting supply of fractional 
\mbox{Calabi--Yau} 
categories arising from semiorthogonal decompositions of Fano varieties. 
For instance, by \mbox{\cite[Corollary~4.1]{kuznetsov2015calabi}}, for every smooth hypersurface 
$X \subset \bP^n$ of degree $d \leq n$ there is a semiorthogonal decomposition 
\begin{equation}
\label{CX}
\Perf(X) = \langle \cC_X, \cO_X, \cO_X(1), \dots, \cO_X(n-d) \rangle, 
\end{equation}
where $\cC_X$ is fractional Calabi--Yau. 
Such categories have recently come into focus as powerful 
tools for studying the geometry of the original variety. 
The most prominent case is that of a cubic fourfold $X \subset \bP^5$, where among other things 
$\cC_X$ has been connected to the longstanding question of whether $X$ is rational \cite{kuznetsov2010derived, addington-thomas}, 
and has been used to give a new proof of the Torelli theorem \cite{huy-ren, BLMS} and new constructions of 
hyperk\"{a}hler varieties \cite{families-stability}. 
The categories $\cC_X$ attached to other varieties are also expected to encode rich information, but 
there remains much to be explored. 
One of the first steps toward understanding such a category is to compute its homological invariants, and 
in particular its Hochschild cohomology. 

In general, it is a difficult problem to compute the Hochschild cohomology of a category~$\cC$. 
If $\cC$ is Calabi--Yau, however, the situation simplifies. 
In this case, $\HH^\bullet(\cC)$ is isomorphic as a graded vector space, up to a shift, to 
the \emph{Hochschild homology} $\HH_{\bullet}(\cC)$ (see Definition~\ref{definition-hochschild-homology}). 
The problem of computing $\HH_\bullet(\cC)$ is often much more tractable 
than computing $\HH^\bullet(\cC)$, because Hochschild homology is additive 
under semiorthogonal decompositions (see~\cite{kuznetsov2009hochschild}). 

If $\cC$ is fractional Calabi--Yau, one would also like an effective mechanism 
for analyzing Hochschild cohomology in terms of Hochschild homology. 
This is what we achieve for a large class of fractional Calabi--Yau categories. 
Namely, suppose $\cC$ admits an autoequivalence $\sigma$ generating 
a $\bZ/q$-action, such that $\rS_{\cC} \cong \sigma \circ [n]$. 
Again, there are many such categories, e.g. for $q=2$ there are infinitely 
many $(d,n)$ such that $\cC_X$ in~\eqref{CX} satisfies this condition. 
If $q$ is coprime to the characteristic of $\bC$, then the invariant category $\cC^{\bZ/q}$ 
is Calabi--Yau (Lemma~\ref{corollary-fCY-invariants}),  
hence $\HH^\bullet(\cC^{\bZ/q})$ is controlled by $\HH_{\bullet}(\cC^{\bZ/q})$, 
and Theorem~\ref{theorem-intro} expresses $\HH^\bullet(\cC)^{\bZ/q}$ 
as a summand of $\HH^\bullet(\cC^{\bZ/q})$. 
In fact, we prove the following result of independent interest, 
which implies $\HH^\bullet(\cC)^{\bZ/q} = \HH^\bullet(\cC)$. 

\begin{proposition}
\label{proposition-serre-HC}
Let $\cC$ be a proper $\bC$-linear 
which admits a Serre functor~$\rS_{\cC}$. 
Then the induced map $\rS_{\cC*} \colon \HH^\bullet(\cC) \to \HH^\bullet(\cC)$ is the identity. 
\end{proposition}

Thus, in the situation above we obtain a splitting of $\HH^\bullet(\cC^{\bZ/q})$ 
with $\HH^\bullet(\cC)$ as a summand. 
When $q = 2$, the complementary summand is entirely controlled 
by the Hochschild homology of $\cC$  
(while for $q>2$ there are other contributions). 

\begin{corollary}
\label{corollary-Z2-HC-intro}
Let $\cC$ be a proper $\bC$-linear 
such that $\rS_{\cC} = \sigma \circ [n]$ is a Serre functor, where $n$ is an integer 
and $\sigma$ is the autoequivalence corresponding to the generator of a 
$\bZ/2$-action on~$\cC$. 
Assume the characteristic of $\bC$ is not $2$. 
Then there is an isomorphism 
\begin{equation}
\label{equation-Z2-HC}
\HH^\bullet(\cC^{\bZ/2}) \cong 
\HH^{\bullet}(\cC) \oplus (\HH_{\bullet}(\cC)^{\bZ/2}[-n]). 
\end{equation}
\end{corollary}

Hence, in the situation of Corollary~\ref{corollary-Z2-HC-intro}, the 
computation of $\HH^\bullet(\cC)$ reduces to an often more tractable question 
about Hochschild homology. 
As an illustration of this method, we prove the following result, 
which identifies the infinitesimal deformation spaces of a 
quartic fourfold $X \subset \bP^5$ and the category 
$\cC_X$ defined by~\eqref{CX}.

\begin{proposition}
\label{proposition-quartic-deformation}
Let $X \subset \bP^5$ be a smooth quartic fourfold over $\bC$, 
and assume $\characteristic(\bC) \neq 2, 3$. 
Then $\dim \HH^2(\cC_X) = 90$ and the natural map 
$\rH^1(X, \rT_X) \to \HH^2(\cC_X)$ is an isomorphism. 
\end{proposition}

For another application of Corollary~\ref{corollary-Z2-HC-intro} to computing 
Hochschild cohomology, we refer to~\cite[Proposition~2.12]{GM-derived-categories}. 

\subsection*{Organization of the paper} 
In \S\ref{section-stable-cats}, we briefly review some facts about $\bC$-linear categories. 
In \S\ref{section-group-actions}, we study finite group actions on such 
categories, and in particular prove that the norm functor is an equivalence 
when the order of the group is invertible in $\bC$. 
Then in \S\ref{section-main-theorem} we prove Theorem~\ref{theorem-intro}, 
in \S\ref{section-serre-HC} we prove Proposition~\ref{proposition-serre-HC},  
and in \S\ref{section-fCY} we prove Corollary~\ref{corollary-Z2-HC-intro} 
and Proposition~\ref{proposition-quartic-deformation}. 

\subsection*{Acknowledgements} 
This paper greatly benefitted from discussions with Alexander Efimov, Valery Lunts, Jacob Lurie, 
and especially Akhil Mathew. 
I would also like to heartily thank Sasha Kuznetsov, for inspiring conversations during which many of 
the ideas in this paper were discovered. 
The original impetus for this paper was the application to GM varieties in our joint work \cite{GM-derived-categories}. 

%%%%%%%%%%%%%%%%%%%%%%%%%%%%%%%%%%%%%%%%%%%%%%%%%%

\section{Preliminaries on $\bC$-linear categories} 
\label{section-stable-cats}
In this section we spell out our conventions on $\bC$-linear categories and 
summarize some of the key points of the theory. 
For background on $\infty$-categories, see~\cite{lurie-HTT, lurie-HA, lurie-SAG}, 
or the survey~\cite[Chapter I.1]{gaitsgory-DAG}. 

\subsection{Small $\bC$-linear categories}
\label{subsection-small-cats}
Let $\Vectfd{\bC}$ denote the $\infty$-category of finite 
complexes of finite-dimensional $\bC$-vector spaces. 
The category $\Vectfd{\bC}$ is stable and has a natural symmetric monoidal structure. 
We use the term \emph{$\bC$-linear category} to mean a 
small idempotent-complete stable $\infty$-category $\cC$ equipped with a module structure 
over $\Vectfd{\bC}$, such that the action functor 
\begin{align*}
\Vectfd{\bC} \times \cC & \to \cC \\
(V, C) & \mapsto V \otimes C
\end{align*}
is exact in both variables. 
All of the functors between $\bC$-linear categories 
considered below will be $\bC$-linear and exact, so we often omit these adjectives. 

A reader unfamiliar with the above language will not lose much by 
thinking of $\cC$ as a small pretriangulated DG category over $\bC$. 
In fact, the theory of such DG categories is equivalent in a precise sense to that of $\bC$-linear categories 
as defined above, see \cite{DG-vs-stable}. 
We note that the homotopy category $\rh\cC$ of a $\bC$-linear category $\cC$ 
is indeed a $\bC$-linear triangulated category. 

Nonetheless, the theory of $\bC$-linear categories and $\infty$-categories 
has several technical advantages over the classical theory of DG categories. 
First, due to the foundations set up in~\cite{lurie-HTT, lurie-HA}, the theory of 
$\infty$-categories is a robust generalization of ordinary category theory. 
In particular, there are $\infty$-categorical notions of commutativity of a diagram, 
and of limits and colimits, which are formally very similar to the corresponding notions 
for ordinary categories. 
Another important feature is that the collection of all $\bC$-linear categories (with morphisms between them the exact $\bC$-linear functors) 
can be organized into an $\infty$-category $\StCat$, which admits small limits and colimits (cf. \cite[\S2.1]{galois-akhil}). 

\subsubsection{Examples}
Given a scheme $X$ over $\bC$, 
there is a $\bC$-linear category $\Perf(X)$ whose 
homotopy category is the usual derived category of perfect complexes 
(see for instance \cite{bzfn}). 
For this paper, the motivating example of a $\bC$-linear category 
is a category appearing as a semiorthogonal component in $\Perf(X)$, 
e.g. the category $\cC_X$ defined by~\eqref{CX}. 
For a discussion of semiorthogonal decompositions in the context of 
stable $\infty$-categories, see~\cite[\S7.2]{lurie-SAG} or \cite[\S3]{NCHPD}. 
We note that giving a semiorthogonal decomposition of a $\bC$-linear category 
$\cC$ is equivalent to giving a semiorthogonal decomposition 
of the triangulated category $\rh\cC$. 

\subsubsection{Mapping objects}
For objects $C,D \in \cC$ in an $\infty$-category $\cC$, we 
write $\Map_{\cC}(C,D)$ for the space of maps from $C$ to $D$. 
Let $\Vect_\bC$ denote the $\infty$-category of complexes of $\bC$-vector spaces. 
If $\cC$ has the structure of a $\bC$-linear category, then there is a 
mapping object $\ulMap_{\cC}(C,D) \in \Vect_\bC$ characterized by equivalences 
\begin{equation}
\label{internal-hom}
\Map_{\Vect_\bC}(V, \ulMap_{\cC}(C,D)) \simeq \Map_{\cC}(V \otimes C, D) 
\end{equation}
for $V \in \Vectfd{\bC}$, see e.g. \cite[\S2.3.1]{NCHPD}. 

\subsection{Large $\bC$-linear categories}
\label{subsection-large-cats}
We review here the ``large'' version of $\bC$-linear categories. 
The only time this material will be needed is in the proof of 
Proposition~\ref{proposition-norm-equivalence} and its attendant lemmas. 

The $\infty$-category $\Vect_{\bC}$ of complexes of $\bC$-vector spaces has a natural 
symmetric monoidal structure. 
A \emph{presentable $\bC$-linear category} is a presentable stable 
$\infty$-category $\calD$ equipped with a module structure over $\Vect_{\bC}$. 
Recall that a presentable $\infty$-category is one which admits small colimits and 
satisfies a mild set-theoretic condition --- roughly, that $\calD$ is generated under 
sufficiently filtered colimits by a small subcategory (see \cite[Chapter 5]{lurie-HTT} for details). 
As a basic example, given a scheme $X$ over $\bC$, there is a presentable 
$\bC$-linear category $\QCoh(X)$, whose homotopy category is 
the usual unbounded derived category of quasi-coherent sheaves 
(see for instance \cite{bzfn}). 

There is an $\infty$-category $\PrCat$ whose objects are the presentable 
$\bC$-linear categories and whose morphisms are the colimit 
preserving $\bC$-linear functors.  
Just as $\StCat$, the category $\PrCat$ admits small limits and colimits. 
Moreover, given any $\calD \in \PrCat$ and objects $C, D \in \calD$, there 
is a mapping object $\ulMap_{\calD}(C,D) \in \Vect_\bC$ characterized by 
equivalences as in~\eqref{internal-hom}, where now $V$ is allowed to be 
any object in $\Vect_{\bC}$. 

\begin{remark}
In the literature, the term ``$\bC$-linear stable $\infty$-category'' 
is often taken to mean a \emph{presentable} $\bC$-linear category in the sense described above. 
We have reserved the term ``$\bC$-linear category'' for the small version of these 
categories, because we will almost exclusively deal with categories of this type. 
\end{remark}

The small and presentable versions of $\bC$-linear categories are 
related via the operation of Ind-completion. Namely, there is a functor 
\begin{equation*}
\Ind\colon \StCat \to \PrCat
\end{equation*}
which takes $\cC \in \StCat$ to its \emph{Ind-completion} $\Ind(\cC) \in \PrCat$. 
Roughly, $\Ind(\cC)$ is obtained from $\cC$ by freely adjoining all filtered colimits. 

\begin{remark}
If $X$ is a quasi-compact separated scheme over $\bC$,  
then there is an equivalence $\Ind(\Perf(X)) \simeq \QCoh(X)$. 
Indeed, by \cite[Proposition 3.19]{bzfn} the scheme $X$ is perfect in 
the sense of \cite[Definition 3.2]{bzfn}, and hence the stated equivalence holds. 
\end{remark}

For the details of Ind-completion, see \cite[Chapter 5]{lurie-HTT} or 
\cite[Chapter I.1, \S7.2]{gaitsgory-DAG}. 
All that we need for our purposes are the following facts. 

\begin{lemma} Ind-completion satisfies the following properties: 
\label{lemma-ind}
\begin{enumerate}
\item \label{ind-1} The functor $\Ind\colon \StCat \to \PrCat$ commutes with colimits. 
\item \label{ind-2} There is a natural fully faithful functor $\cC \hookrightarrow \Ind(\cC)$ of 
$\Vectfd{\bC}$-module categories whose essential image is the subcategory 
$\Ind(\cC)^c \subset \Ind(\cC)$ of compact objects, and which therefore induces an 
equivalence $\cC \simeq \Ind(\cC)^c$. 
\end{enumerate}
\end{lemma}

%%%%%%%%%%%%%%%%%%%%%%%%%%%%%%%%%%%%%%%%%%%

\section{Group actions on categories}
\label{section-group-actions}

In this section, we discuss group actions on $\bC$-linear categories. 
Besides recalling the basic definitions, we prove that the norm functor 
from the category of coinvariants to the category of invariants for the 
action of a finite group is an equivalence, provided the order of the 
group is prime to the characteristic of $\bC$ (Proposition~\ref{proposition-norm-equivalence}). 

\subsection{Group actions} 
Let $G$ be a finite group. We denote by $BG$ the classifying space of 
$G$, regarded as an $\infty$-category (i.e. $BG$ is the nerve of the 
ordinary category with a single object whose endomorphisms are given 
by the group $G$). 

\begin{definition}
\label{definition-group-action}
Let $\calD$ be an $\infty$-category, and let $C \in \calD$ be an object. 
An \emph{action} of $G$ on~$C$ is a functor $\phi\colon BG \to \calD$ 
which carries the unique object $* \in BG$ to $C \in \calD$. 
Given such an action, the \emph{$G$-invariants} $C^G$ and 
\emph{$G$-coinvariants} $C_G$ are defined by 
\begin{equation*}
C^G = \lim(\phi) \qquad \text{and} \qquad 
C_G = \colim(\phi),
\end{equation*}
provided the displayed limit and colimit exist. 
In this case, we denote by $p \colon C^G \rightarrow C$ 
and $q \colon C \to C_G$ the canonical morphisms. 
\end{definition}

\begin{remark}
\label{remark-opposite-action}
What we have called an action of $G$ could more precisely be 
called a \emph{left action} of $G$. 
There is also a notion of a \emph{right action} of $G$ on an object $C \in \calD$, 
namely, a functor $\psi\colon BG^\opp \to \calD$ from the opposite category of 
$BG$ which carries the basepoint to $C$. 
Note that any left action $\phi\colon BG \to \calD$ on $C$ gives rise 
to a right action by composing with the equivalence $BG^{\opp} \simeq BG$ 
induced by inversion in~$G$. 
\end{remark}

To relate the above definition to the classical notion of a group action, 
note that $\phi$ specifies for each $g \in G$ (thought of as an endomorphism 
of the basepoint $* \in BG$) an equivalence ${\phi_g\colon C \to C}$. 
The data of the entire action functor $\phi\colon BG \to \calD$ then specifies  
certain compatibilities among the $\phi_g$. 

We will be particularly interested in the case of Definition~\ref{definition-group-action} 
where $\calD = \StCat$, i.e. where $G$ acts on a $\bC$-linear category $\cC \in \StCat$.
Since $\StCat$ admits limits and colimits, the $G$-invariants 
$\cC^G$ and coinvariants $\cC_G$ always exist in this situation. 

The objects and morphisms of $\cC^G$ can be described in relatively 
concrete terms, as follows. 
By the universal property of the projection functor $\cC^G \to \cC$, an 
object of $\wtilde{C} \in \cC^G$ corresponds to an object $C \in \cC$ together 
with a \emph{linearization}, i.e. a family of equivalences \mbox{$\ell_g\colon C \to \phi_g(C)$} for $g \in G$ 
satisfying certain compatibilities. 
By abuse of notation, given an object $C \in \cC$ with a linearization, 
we will often denote by the same symbol $C$ the corresponding object 
of $\cC^G$.  

Given another object $\wtilde{D} \in \cC^G$ corresponding to 
an object $D \in \cC$ with linearization maps \mbox{$m_g\colon D \to \phi_g(D)$}, there is a natural action of 
$G$ on $\ulMap_{\cC}(C, D)$, i.e. there is a functor $\rho\colon BG \to \Vect_{\bC}$ sending 
the basepoint $* \in BG$ to $\ulMap_{\cC}(C,D) \in \Vect_{\bC}$. 
Concretely, the image of $f\colon C \to D$
under $\rho_g\colon \ulMap_{\cC}(C, D) \to \ulMap_{\cC}(C, D)$ is 
determined by the commutative diagram 
\begin{equation*}
\xymatrix{
C \ar[rr]^{\rho_g(f)} \ar[d]_{\ell_g} && D \ar[d]^{m_g} \\ 
\phi_g(C) \ar[rr]^{\phi_g(f)} && \phi_g(D)
}
\end{equation*}
The morphisms in the category $\cC^G$ are then described by the formula 
\begin{equation}
\label{mapping-space-invariants}
\ulMap_{\cC^G}(\wtilde{C}, \wtilde{D}) \simeq \ulMap_{\cC}(C, D)^G. 
\end{equation}

\subsection{Induction and restriction} 
Let $G$ be a finite group and $H \subset G$ a subgroup. Recall that there are natural induction 
and restriction functors between the classical categories of linear representations 
of $H$ and~$G$, and that these functors are mutually left and right adjoint. 
There is an analogous relation between the categories of invariants for 
group actions on a $\bC$-linear category. 

Namely, let $\cC$ be a $\bC$-linear category with a $G$-action. 
The functor $BH \to BG$ induces an $H$-action on $\cC$. 
We denote by 
\begin{equation*}
\Res^G_H\colon \cC^G \to \cC^H 
\end{equation*}
the resulting \emph{restriction functor} from $G$-invariants to $H$-invariants. 
Further, by choosing representatives for the cosets of $G/H$, we obtain a functor 
\begin{equation*}
\textstyle \bigoplus \limits_{[g] \in G/H} \phi_g \colon \cC \to \cC,
\quad 
C \mapsto \bigoplus \limits_{[g] \in G/H} \phi_g(C).
\end{equation*}
The precomposition with the projection $\cC^H \to \cC$ lifts to the 
\emph{induction functor} 
\begin{equation*}
\Ind_H^G \colon \cC^H \to \cC^G, 
\end{equation*}
which fits into a commutative diagram
\begin{equation*}
\xymatrix{
\cC^H \ar[rr]^{\Ind_H^G} \ar[d] && \cC^G \ar[d] \\ 
\cC \ar[rr]^{\bigoplus \limits_{[g] \in G/H} \phi_g} && \cC
}
\end{equation*}
If $H$ is the trivial group, we write $\Av \colon \cC \to \cC^G$ for the induction 
functor and call it the \emph{averaging functor} 
(also known as the \emph{inflation functor}). 

In the setting of triangulated categories with a group action, it is well-known that the 
induction and restriction functors are mutually left and right adjoint 
(see~\cite[Lemma~3.8]{elagin-equivariant} for the case where $H$ is the trivial group). 
Similarly, we have the following. 

\begin{lemma}
Let $\cC$ be a $\bC$-linear category with an action by a finite group $G$, 
and let $H \subset G$ be a subgroup. 
Then the functors $\Ind_H^G$ and $\Res_H^G$ are mutually 
left and right adjoint. 
\end{lemma}

\subsection{The norm functor} 
Let $G$ be a finite group and $V$ be a (ordinary) $\bC$-vector space with a $G$-action.  
Then there is a \emph{norm map} $\Nm\colon V_G \to V^G$, 
induced by the map $V \to V$ given by $x \mapsto \sum_{g \in G} g(x)$. 
If the order of $G$ is invertible in $\bC$, it is easy to see that the norm map is an isomorphism. 
Under the same assumption, the analogous statement holds in the $\infty$-categorical 
setting where $V \in \Vect_\bC$ and $G$ acts on $V$: the norm map $V_G \to V^G$ is an equivalence. 
Our goal below is to describe what happens when $V$ is replaced by a $\bC$-linear category. 

Let $\cC$ be a $\bC$-linear category with a $G$-action. 
By the universal properties of $\cC_G$ and $\cC^G$, the functor 
$\textstyle \bigoplus_{g \in G} \phi_g  \colon \cC \to \cC$ induces the \emph{norm functor} 
\begin{equation*}
\Nm\colon \cC_G \to \cC^G, 
\end{equation*}
which is characterized by the 
existence of a factorization
\begin{equation*}
\xymatrix{
\cC \ar[rr]^{\bigoplus_{g \in G} \phi_g} \ar[d]_{q} && \cC \\ 
\cC_G \ar[rr]^{\Nm} && \cC^G \ar[u]_{p}
}
\end{equation*}
Note that the composition $\Nm \circ q \colon \cC \to \cC^G$ is nothing 
but the averaging functor $\Av$. 
We aim to prove the following. 

\begin{proposition}
\label{proposition-norm-equivalence}
Let $\cC$ be a $\bC$-linear category with an action by a finite group~$G$. 
Assume the order of $G$ is coprime to the characteristic of $\bC$. 
Then the norm functor \mbox{$\Nm \colon \cC_G \to \cC^G$} is an equivalence. 
\end{proposition}

The proof of Proposition~\ref{proposition-norm-equivalence} takes the rest of this section. 
Our strategy is to first prove an analogous assertion for 
presentable $\bC$-linear categories. 
First note that if $\calD$ is such a category with a $G$-action, then the same 
construction as above provides a norm functor $\Nm\colon \calD_G \to \calD^G$. 
In this setting, we have the following analogue of Proposition~\ref{proposition-norm-equivalence}, 
which holds even without any assumption on the characteristic of $\bC$. 

\begin{lemma}
\label{lemma-presentable-norm-equivalence}
Let $\calD$ be a presentable $\bC$-linear category with an action by a finite group $G$. 
Then the norm functor $\Nm\colon \calD_G \to \calD^G$ is an equivalence. 
\end{lemma}

\begin{proof}
We use the following fundamental relation between colimits and limits in 
the presentable setting. 
Let $F\colon I \to \PrCat$ be a functor from a small $\infty$-category $I$ to $\PrCat$; 
we think of $F$ as a diagram of categories 
$F(i) = \cC_i$ for $i \in I$. 
There is a functor $G\colon I^\op \to \PrCat$ obtained by ``passing to right adjoints'', 
such that for every $i \in I$ we have $G(i) = \cC_i$, and for every morphism 
$a\colon i \to j$ in $I$ regarded as a morphism $j \to i$ in $I^\op$, we have  
$G(a) = F(a)^! \colon \cC_j \to \cC_i$ where $F(a)^!$ is the right adjoint to $F(a)$ (see \cite[Chapter I.1, \S2.4]{gaitsgory-DAG}). 
Then the key fact is that 
there is an equivalence $\colim(F) \simeq \lim(G)$, induced by the left adjoints  
$p_i^*\colon \cC_i \to \lim(G)$ to the natural functors $p_i\colon \lim(G) \to \cC_i$ for $i \in I$ 
(see~\cite[Chapter I.1, \S2.5.8]{gaitsgory-DAG}). 

Let us apply this to the functor $\phi\colon BG \to \PrCat$ encoding the action of 
$G$ on $\calD$. The functor $\psi\colon BG^\op \to \PrCat$ obtained by passing to 
right adjoints is nothing but the right $G$-action induced by $\phi$, as described 
in Remark~\ref{remark-opposite-action}. In particular, since $\psi$ differs from  
$\phi$ by composition with the equivalence $BG^\op \simeq BG$, we see that 
$\lim(\psi) \simeq \lim(\phi) = \calD^G$. Hence applying the key fact from above, 
we find that there is an equivalence $\calD_G \simeq \calD^G$. Moreover, this 
equivalence is induced by the left adjoint 
to $p\colon \calD^G \to \calD$, i.e. by the 
averaging functor $\Av\colon \calD \to \calD^G$, and hence is given by the norm functor. 
\end{proof}

We will also need the following result. We use the notation $\calD^c$ for the 
full subcategory of compact objects of an $\infty$-category $\calD$. 
\begin{lemma}
\label{lemma-invariants-cpt-objects}
Let $\cC$ be a $\bC$-linear category with an action by a finite group~$G$. 
Then the natural fully faithful functor $\cC \hookrightarrow \Ind(\cC)$ induces an equivalence 
$\cC_G \simeq (\Ind(\cC)_G)^c$. 
If the order of $G$ is coprime to the characteristic of $\bC$, then $\cC \hookrightarrow \Ind(\cC)$ 
also induces an equivalence $\cC^G \simeq (\Ind(\cC)^G)^c$. 
\end{lemma}

\begin{proof}
For the coinvariants, note that by Lemma~\ref{lemma-ind}\eqref{ind-1} we have 
$\Ind(\cC)_G \simeq \Ind(\cC_G)$. 
Hence by Lemma~\ref{lemma-ind}\eqref{ind-2} we see the functor 
$\cC_G \to \Ind(\cC)_G$ factors through an equivalence $\cC_G \simeq (\Ind(\cC)_G)^c$. 
For the invariants, observe that the natural functor $\cC^G \to \Ind(\cC)^G$ is 
fully faithful, by fully faithfulness of $\cC \hookrightarrow \Ind(\cC)$ combined 
with the description~\eqref{mapping-space-invariants} (which also holds in the presentable 
setting) of mapping spaces in a category of invariants. This realizes 
$\cC^G$ as the full subcategory of $\Ind(\cC)^G$ consisting of objects whose image  
under $\Ind(\cC)^G \to \Ind(\cC)$ is in $\cC \simeq \Ind(\cC)^c \subset \Ind(\cC)$. 
Now the equivalence $\cC^G \simeq (\Ind(\cC)^G)^c$ is a consequence 
of the following lemma. 
\end{proof}

\begin{lemma}
Let $\calD$ be a presentable $\bC$-linear category with 
an action by a finite group $G$. 
Assume the order of $G$ is coprime to the characteristic of $\bC$. 
Then an object of $\calD^G$ is compact if and only if its image under 
$\calD^G \to \calD$ is compact. 
\end{lemma}

\begin{proof}
Let $\wtilde{C}$ be an object of $\calD^G$ and $C$ its image in $\calD$. 
Assume first that $\wtilde{C}$ is compact. Then by adjointness of induction and restriction, 
for $D \in \calD$ we have 
\begin{equation*}
\ulMap_{\calD}(C,D) \simeq \ulMap_{\calD^G}(\wtilde{C}, \Av(D)). 
\end{equation*}
The functor $\Av$ commutes with colimits because it is a left adjoint, 
and $\ulMap_{\calD^G}(\wtilde{C}, -)$ commutes with filtered colimits by compactness of $\wtilde{C}$.
Hence the above equivalence shows that $\ulMap_{\calD}(C,-)$ commutes with 
filtered colimits, i.e. that $C$ is compact. 

Now assume that $C$ is compact. 
Let $p\colon \calD^G \to \calD$ be the projection. 
Then for $\wtilde{D} \in \calD^G$, we have
\begin{equation*}
\ulMap_{\calD^G}(\wtilde{C}, \wtilde{D}) \simeq \ulMap_{\calD}(C,p(\wtilde{D}))^G 
\simeq \ulMap_{\calD}(C,p(\wtilde{D}))_G, 
\end{equation*}
where 
the second equivalence is given by the inverse of the norm map (here we use the 
assumption on the characteristic of $\bC$). 
Note that $p\colon \calD^G \to \calD$ commutes with colimits because it is a left adjoint, 
$\ulMap_{\calD}(C,-)$ commutes with filtered colimits by compactness of $C$, 
and taking $G$-coinvariants commutes with colimits because by definition 
it is given by a colimit. 
Hence the above equivalence shows that $\ulMap_{\calD^G}(\wtilde{C}, -)$ commutes 
with filtered colimits, i.e. that $\wtilde{C}$ is compact. 
\end{proof}

\begin{proof}[Proof of Proposition~\textup{\ref{proposition-norm-equivalence}}]
By Lemma~\ref{lemma-presentable-norm-equivalence} 
the norm functor $\Nm\colon \Ind(\cC)_G \to \Ind(\cC)^G$ is an equivalence, 
and hence so is its restriction to the full subcategories of compact objects.  
But by Lemma~\ref{lemma-invariants-cpt-objects}, this restriction is 
identified with $\Nm\colon \cC_G \to \cC^G$. 
\end{proof}

%%%%%%%%%%%%%%%%%%%%%%%%%%%%%%%%%%%%%%%%

\section{Hochschild cohomology and group invariants}
\label{section-main-theorem}

Our goal in this section is to prove Theorem~\ref{theorem-intro}, 
stated more precisely as Theorem~\ref{theorem-equivariant-HC} below. 
We start by recalling the definition of Hochschild cohomology in our setting. 

\subsection{Hochschild cohomology} 
If $\cC$ and $\calD$ are $\bC$-linear categories, then 
the $\bC$-linear exact functors from $\cC$ to $\calD$ form the objects 
of a $\bC$-linear category $\Fun_\bC(\cC, \calD)$. 

\begin{remark}
\label{remark-bzfn}
Let $X$ and $Y$ be smooth and proper schemes over $\bC$. 
Then by \cite[Theorem~1.2]{bzfn} there is an equivalence of $\bC$-linear categories 
\begin{equation*}
\Perf(X \times Y) \simeq \Fun_{\bC}(\Perf(X), \Perf(Y)) 
\end{equation*} 
which takes an object $\cE \in \Perf(X \times Y)$ to the corresponding Fourier--Mukai 
functor $\Phi_{\cE} \colon \Perf(X) \to \Perf(Y)$. 
\end{remark}

\begin{definition}
\label{definition-hochschild-cohomology}
Let $\cC$ be a $\bC$-linear category, and let $\phi \colon \cC \to \cC$ be an endofunctor. 
The \emph{Hochschild cochain complex of $\cC$ with coefficients in $\phi$} is 
defined as
\begin{equation*}
\HC^\bullet(\cC, \phi) = \ulMap_{\Fun_\bC(\cC, \cC)}(\id_{\cC}, \phi) \in \Vect_{\bC}. 
\end{equation*}
The \emph{Hochschild cohomology $\HH^{\bullet}(\cC, \phi)$ of $\cC$ with coefficients in $\phi$}  
is the cohomology of this complex. 
In case $\phi = \id_{\cC}$, we write 
\begin{equation*}
\HC^\bullet(\cC) = \HC^\bullet(\cC, \id_{\cC}) \quad \text{and} \quad
\HH^{\bullet}(\cC) = \HH^{\bullet}(\cC, \id_{\cC}) , 
\end{equation*} 
and call these the \emph{Hochschild cochain complex} and \emph{Hochschild cohomology} 
of $\cC$. 
\end{definition}

\begin{remark}
There is a natural algebra structure on the Hochschild cohomology $\HH^\bullet(\cC)$, 
induced by composition in $\ulMap_{\Fun_\bC(\cC, \cC)}(\id_{\cC}, \id_{\cC})$.
\end{remark}

Hochschild cohomology is not functorial with respect to arbitrary functors of  
$\bC$-linear categories. 
But, of course, it is functorial with respect to equivalences. 
Namely, if $\Phi\colon \cC \to \calD$ is an equivalence, conjugation by $\Phi$ induces an isomorphism 
\begin{equation*}
\Phi_*\colon \HH^\bullet(\cC) \xrightarrow{\, \sim \,} \HH^\bullet(\calD). 
\end{equation*}
Explicitly, given $a \colon \id_{\cC} \to \id_{\cC}[n]$ representing an element of 
$\HH^n(\cC)$, its image $\Phi_*(a)$ is determined by the commutative diagram 
\xyoption{rotate}
\begin{equation*}
\xymatrix{
\id_{\calD} \ar[rr]^{\Phi_*(a)} \ar[d]_[@!-90]{\sim} && \id_{\calD}[n] \ar[d]^[@!-90]{\sim} \\
\Phi \circ \id_{\cC} \circ \Phi^{-1} \ar[rr]^{\Phi  a  \Phi^{-1}} && 
\Phi \circ \id_{\cC}[n] \circ \Phi^{-1}
}
\end{equation*}

\subsection{The main theorem}
Note that given a $\bC$-linear category $\cC$ with an action by a finite group $G$, 
there is an induced action of $G \times G$ on $\Fun_{\bC}(\cC, \cC)$. 
Concretely, an element $(g_1, g_2) \in G \times G$ acts on $\Fun_{\bC}(\cC, \cC)$ by sending 
$F\colon \cC \to \cC \in \Fun_{\bC}(\cC, \cC)$ to $\phi_{g_2} \circ F \circ \phi_{g_1}^{-1}$. 
Via the diagonal embedding $G \subset G \times G$, this restricts to the conjugation 
action of $G$ on $\Fun_{\bC}(\cC, \cC)$. 

\begin{theorem}
\label{theorem-equivariant-HC}
Let $\cC$ be a $\bC$-linear category 
with an action by a finite group $G$. 
Assume the order of $G$ is coprime to the characteristic of $\bC$. 
Then there is an isomorphism
\begin{equation*}
\HH^\bullet(\cC^G) \cong 
\textstyle \left( \bigoplus_{g \in G} \HH^\bullet(\cC, \phi_g) \right)^G , 
\end{equation*}
where $\phi_g\colon \cC \to \cC$ is the autoequivalence corresponding to $g \in G$, and 
the $G$-action on the right side is induced by the conjugation action of 
$G$ on $\Fun_{\bC}(\cC, \cC)$. 
\end{theorem}

\begin{remark}
The action of $G$ on the right side of the isomorphism of Theorem~\ref{theorem-equivariant-HC} 
preserves the term $\HH^\bullet(\cC)$ corresponding to $g = 1$; hence  
$\HH^\bullet(\cC)^G$ appears as a summand of $\HH^\bullet(\cC^G)$, as stated in 
Theorem~\ref{theorem-intro} from the introduction. 
We note that  inclusion $\HH^\bullet(\cC)^G \hookrightarrow \HH^\bullet(\cC^G)$ is compatible with the algebra structure 
on Hochschild cohomology, i.e. it realizes $\HH^\bullet(\cC)^G$ as a subalgebra of $\HH^\bullet(\cC^G)$. 
This follows from the proof given below.
\end{remark}

\begin{remark}
Theorem~\ref{theorem-equivariant-HC} holds verbatim for a presentable 
$\bC$-linear category~$\calD$ in place of $\cC$, with the same proof given below. 
\end{remark}

We will need the following lemma for the proof of the theorem. 
\begin{lemma}
\label{lemma-equivariant-functor-category}
Let $\cC$ be a $\bC$-linear category with an action 
by a finite group $G$. 
Assume the order of $G$ is coprime to the characteristic of $\bC$. 
Then there is an equivalence 
\begin{equation*}
\Fun_{\bC}(\cC^G, \cC^G) \simeq 
\Fun_{\bC}(\cC, \cC)^{G \times G} 
\end{equation*}
under which the identity $\id_{\cC^G}$ corresponds to 
the functor $\bigoplus_{g \in G} \phi_g  \colon \cC \to \cC$ 
\textup{(}with the natural $G \times G$-linearization\textup{)}. 
\end{lemma}

\begin{proof}
The norm equivalence $\Nm\colon \cC_G \to \cC^G$ induces 
an equivalence 
\begin{equation*}
\Fun_{\bC}(\cC^G, \cC^G) \simeq \Fun_{\bC}(\cC_G, \cC^G). 
\end{equation*}
The formation of the functor category $\Fun_{\bC}(-,-)$ between 
$\bC$-linear categories takes colimits in the first 
variable to limits, and limits in the second variable to limits. 
Applying this to $\cC_G = \colim_{BG} \cC$ and 
then $\cC^G = \lim_{BG} \cC$, we find
\begin{equation*}
\Fun_{\bC}(\cC_G, \cC^G) \simeq \Fun_{\bC}(\cC, \cC^G)^G 
\simeq (\Fun_{\bC}(\cC, \cC)^G)^G. 
\end{equation*}
In the final term, the outer $G$-action is induced by the action of $G$ 
on the first copy of $\cC$, and the inner $G$-action by the action 
of $G$ on the second copy of $\cC$. 
Thus, the outer $G$-action is induced by the restriction of the 
$G \times G$-action on $\Fun_{\bC}(\cC, \cC)$ to the first factor, 
and the inner $G$-action is identified with the restriction of the $G \times G$-action 
to the second factor. It follows that 
\begin{equation*}
(\Fun_{\bC}(\cC, \cC)^G)^G \simeq \Fun_{\bC}(\cC, \cC)^{G \times G}.
\end{equation*} 
All together this proves the equivalence stated in the lemma, 
and tracing through the intermediate equivalences above  
shows that $\id_{\cC^G}$ corresponds to $\bigoplus_{g \in G} \phi_g$. 
\end{proof}

\begin{proof}[Proof of Theorem~\textup{\ref{theorem-equivariant-HC}}]
By the definition of the Hochschild cochain complex combined with  
Lemma~\ref{lemma-equivariant-functor-category}, we have 
\begin{equation*}
\HC^\bullet(\cC^G) = \ulMap_{\Fun_{\bC}(\cC^G, \cC^G)}(\id_{\cC^G}, \id_{\cC^G}) 
\simeq \ulMap_{\Fun_{\bC}(\cC, \cC)^{G \times G}}(\textstyle \bigoplus_{g \in G} \phi_g, \bigoplus_{g \in G} \phi_g). 
\end{equation*}
Recall that $G$ acts on $\Fun_{\bC}(\cC, \cC)$ via the restriction along 
the diagonal embedding $G \subset G \times G$. 
Notice that the induction of $\id_{\cC} \in \Fun_{\bC}(\cC, \cC)^G$ 
along the diagonal is  
\begin{equation*}
\Ind_G^{G \times G}(\id_{\cC}) = \textstyle \bigoplus_{g \in G} \phi_g 
\in \Fun_{\bC}(\cC,\cC)^{G \times G}. 
\end{equation*}
Hence by adjointness of induction and restriction, we can rewrite the above expression as 
\begin{align*}
\HC^\bullet(\cC^G) & \simeq 
\textstyle \ulMap_{\Fun_{\bC}(\cC, \cC)^{G}}(\id_{\cC}, \bigoplus_{g \in G} \phi_g) \\ 
& \simeq 
\textstyle \left( \bigoplus_{g \in G} \ulMap_{\Fun_{\bC}(\cC, \cC)}(\id_{\cC}, \phi_g) \right)^G   \\
& = 
\textstyle  \left( \bigoplus_{g \in G} \HC^\bullet(\cC, \phi_g) \right)^G . 
\end{align*}
By our assumption on the characteristic of $\bC$, the operation of 
taking group invariants commutes with taking cohomology. So by 
taking cohomology, the theorem follows. 
\end{proof}

We end this section by explaining how Theorem~\ref{theorem-equivariant-HC} can be 
used to reprove the orbifold HKR decomposition 
from~\cite[Corollary 1.17(3)]{orbiHKR-caldararu}. 

\begin{corollary}
\label{corollary-orbifold-HKR}
Let $X$ be a smooth proper scheme over $\bC$ with an action by a finite group $G$. 
Assume that $\characteristic(\bC) = 0$, or that $\characteristic(\bC)$ is coprime to 
the order of $G$ and $\characteristic(\bC) \geq \dim(X)$. 
For $g \in G$ let $X^g$ denote the fixed locus of $g$ in $X$, 
and let $c_g$ denote the codimension of $X^g$ in $X$. 
Set $\HH^\bullet([X/G]) = \HH^\bullet(\Perf([X/G]))$. 
Then there is an isomorphism 
\begin{equation*}
\HH^n([X/G]) \cong \left( \bigoplus_{g \in G} \bigoplus_{p+q = n} \rH^{q-c_g}(X^g, \wedge^p \rT_{X^g} \otimes \det(\rN_{X^g/X}) ) \right)^G . 
\end{equation*} 
\end{corollary} 

\begin{proof}
All pushforward and pullback functors considered in this proof are by convention derived. 
First note that each $X^g$ is smooth by \cite[Proposition 3.4]{edixhoven}. 
In view of the equivalence $\Perf([X/G]) \simeq \Perf(X)^G$, 
by Theorem~\ref{theorem-equivariant-HC} it suffices to show that 
\begin{equation}
\label{orbifold-HKR-goal} 
\HH^n(\Perf(X), \phi_g) \cong \bigoplus_{p+q = n} \rH^{q-c_g}(X^g, \wedge^p \rT_{X^g} \otimes \det(\rN_{X^g/X})) . 
\end{equation} 
Under the equivalence 
\begin{equation*}
\Fun_{\bC}(\Perf(X), \Perf(X)) \simeq \Perf(X \times X)
\end{equation*}
of Remark~\ref{remark-bzfn}, the functor $\phi_g = g_* \colon \Perf(X) \to \Perf(X)$ maps to the object 
$\gamma_{g*}(\cO)$ where $\gamma_g \colon X \to X \times X$ is the graph of $g \colon X \to X$. 
Hence we have 
\begin{equation*}
\HH^\bullet(\Perf(X), \phi_g) \cong \Ext_{X \times X}^\bullet(\Delta_{X*}(\cO), \gamma_{g*}(\cO)) 
\end{equation*} 
where $\Delta_X \colon X \to X \times X$ is the diagonal. 
As computed in the proof of \cite[Theorem 3.2]{ganter}, we have an isomorphism 
\begin{equation*}
 \Ext_{X \times X}^\bullet(\Delta_{X*}(\cO), \gamma_{g*}(\cO))  \cong 
 \Ext_{X^g \times X^g}^{\bullet-c_g}(\Delta_{X^g*}(\cO), \Delta_{X^g*}(\det(\rN_{X^g/X}))) 
\end{equation*} 
where $\Delta_{X^g} \colon X^g \to X^g \times X^g$ is the diagonal. 
By the sheafy HKR decomposition (see \cite[Theorem 4.8]{yekutieli}, or \cite[Corollary 1.5]{HKRcharp} if $\characteristic(\bC) = \dim(X)$) we have an isomorphism 
$\Delta_{X^g}^*\Delta_{X^g*}(\cO) \simeq \bigoplus_{p \geq 0} \Omega^p_{X^g}[p]$, 
hence by adjunction we obtain 
\begin{align*}
\Ext_{X^g \times X^g}^{\bullet-c_g}(\Delta_{X^g*}(\cO), \Delta_{X^g*}(\det(\rN_{X^g/X})))  & \cong 
\bigoplus_{p \geq 0} \Ext_{X^g \times X^g}^{\bullet-c_g}(\Omega^p_{X^g}[p], \det(\rN_{X^g/X}) ) \\ 
& \cong 
\bigoplus_{p \geq 0} \rH^{\bullet-c_g-p}(X^g, \wedge^p\rT_{X^g} \otimes \det(\rN_{X^g/X}) ) . 
\end{align*} 
Combining the above isomorphisms gives \eqref{orbifold-HKR-goal}, as required. 
\end{proof}

%%%%%%%%%%%%%%%%%%%%%%%%%%%%%%%%%%%%%%%%%%%%%%%%%%%%%%%

\section{The action of a Serre functor on Hochschild cohomology}
\label{section-serre-HC}

Our goal in this section is to prove Proposition~\ref{proposition-serre-HC}. 
We start by recalling the definition of a Serre functor in our setting. 

\subsection{Serre functors}
A $\bC$-linear category $\cC$ is \emph{proper} if for all 
$C,D \in \cC$, the mapping object 
$\ulMap_{\cC}(C,D)$ lies in the essential image of $\Vectfd{\bC} \to \Vect_\bC$, 
i.e. if its total cohomology $\bigoplus_i \rH^i(\ulMap_{\cC}(C,D))$ is finite-dimensional. 
In this situation, we have functors
\begin{align*}
\ulMap_{\cC} \colon & \cC^\op \times \cC \to  \Vectfd{\bC}   &  
\ulMap_{\cC}^\op \colon & \cC^\op \times \cC \to (\Vectfd{\bC})^\op   \\
& (C, D)  \mapsto \ulMap_{\cC}(C,D) &  & (C,D) \mapsto  \ulMap_{\cC}(D,C)
\end{align*}
Note that there is an equivalence $(-)^{\vee}\colon (\Vectfd{\bC})^{\op} \xrightarrow{\sim} \Vectfd{\bC}$ 
given by dualization of complexes. 
Here and below, given an $\infty$-category $\cC$, we denote by $\cC^{\op}$ its opposite category. 

\begin{definition}
Let $\cC$ be a proper $\bC$-linear category. 
A \emph{Serre functor} for $\cC$ is an autoequivalence $\rS_{\cC} \colon \cC \xrightarrow{\sim} \cC$, 
such that there is a commutative diagram 
\begin{equation*}
\xymatrix{
\cC^{\op} \times \cC \ar[rr]^{\id \times \rS_{\cC}} \ar[d]_{\ulMap_{\cC}^\op} && \cC^\op \times \cC \ar[d]^{\ulMap_{\cC}} \\  
(\Vectfd{\bC})^\op \ar[rr]^{(-)^\vee} &&  \Vectfd{\bC}
}
\end{equation*}
\end{definition}

In other words, 
a Serre functor for $\cC$ is characterized by the existence of natural equivalences
\begin{equation*}
 \ulMap_{\cC}(C, \rS_{\cC}(D)) \simeq \ulMap_{\cC}(D, C)^{\vee}
\end{equation*}
for all objects $C, D \in \cC$. 
Note that if $\rS_{\cC}$ is a Serre functor for $\cC$, then it induces an 
autoequivalence of the homotopy category $\rh\cC$, which is a Serre functor 
in the usual sense of triangulated categories, as defined in~\cite{bondal-kapranov}. 

\begin{remark}
If $X$ is a smooth proper scheme over $\bC$, then $\Perf(X)$ 
has a Serre functor given by $F \mapsto F \otimes \omega_X [ \dim(X)]$. 
\end{remark}

\subsection{Action of the Serre functor}

We will need the following observation for the proof of Proposition~\ref{proposition-serre-HC}. 

\begin{lemma}
\label{lemma-adjoint-functoriality}
Let $F, G\colon \cC \to \calD$ be functors between $\bC$-linear categories. 
Let $a \colon F \to G$ be a point of
the mapping space $\Map_{\Fun_\bC(\cC, \calD)}(F, G)$. 
Assume $F$ and $G$ admit right adjoints $F^!$ and $G^!$. 
Let $\eta_F\colon \id_{\cC} \to F^! \circ F$ be the unit of the adjunction between $F$ and $F^!$, 
and let $\epsilon_G\colon G \circ G^! \to \id_{\calD}$ be the counit of the adjunction 
between $G$ and $G^!$. 
Define $a^!\colon G^! \to F^!$ as the composition
\begin{equation*}
a^!\colon G^! = \id_{\cC} \circ G^! \xrightarrow{\, \eta_F G^! \, } F^! \circ F \circ G^! 
\xrightarrow{\, F^! a G^! \,} F^! \circ G \circ G^! \xrightarrow{\, F^! \epsilon_G \,} 
F^! \circ \id_{\calD} = F^!. 
\end{equation*}
Then there is functorially in $C,D \in \cC$ a commutative diagram 
\begin{equation} 
\label{adjoint-diagram}
\vcenter{
\xymatrix{
\ulMap_{\calD}(F(C), D) \ar[d]_[@!-90]{\sim} & \ar[l] \ulMap_{\calD}(G(C), D) \ar[d]^[@!-90]{\sim} \\
\ulMap_{\cC}(C, F^!(D)) & \ar[l] \ulMap_{\cC}(C, G^!(D)) 
}
}
\end{equation}
where the top horizontal arrow is induced by $a\colon F \to G$, 
the bottom horizontal arrow is induced by $a^!\colon G^! \to F^!$, and the vertical arrows are given by adjunction. 
\end{lemma}

\begin{proof}
By Yoneda there exists a morphism $a^! \colon G^! \to F^!$ making~\eqref{adjoint-diagram} commute, so 
all we need to do is check that $a^!$ is given by the claimed formula. 
The morphism $a^!(D) \colon G^!(D) \to F^!(D)$ is given by the image of 
$\id_{G^!(D)} \in \ulMap_{\cC}(G^!(D), G^!(D))$ under the bottom arrow of~\eqref{adjoint-diagram} for $C = G^!(D)$. 
In general if $t \colon C \to G^!(D)$ is a morphism, under the inverse of the right vertical arrow in~\eqref{adjoint-diagram} 
it maps to 
\begin{equation*}
G(C) \xrightarrow{ \, G(t) \,} GG^!(D) \xrightarrow{ \, \epsilon_G(D) \, } D, 
\end{equation*} 
which under the top horizontal arrow in~\eqref{adjoint-diagram} maps to 
\begin{equation*}
F(C) \xrightarrow{\, a(C) \,} G(C) \xrightarrow{ \, G(t) \,} GG^!(D) \xrightarrow{ \, \epsilon_G(D) \, } D, 
\end{equation*} 
which under the left vertical arrow in~\eqref{adjoint-diagram} maps to 
\begin{equation*}
C \xrightarrow{\, \eta_F(C) \,} 
F^!F(C) \xrightarrow{\, F^!a(C) \,} F^!G(C) \xrightarrow{ \, F^!G(t) \,} F^!GG^!(D) \xrightarrow{ \, F^!\epsilon_G(D) \, } F^!(D). 
\end{equation*}
Taking $t = \id_{G^!(D)}$ gives the claimed formula for $a^!$. 
\end{proof} 

It is well known that a Serre functor commutes with all autoequivalences. 
The following elaborates on this by giving a sense in which the commutation 
is functorial with respect to morphisms of autoequivalences. 

\begin{lemma}
\label{lemma-serre-HC}
Let $\cC$ be a proper $\bC$-linear category, 
which admits a Serre functor $\rS_{\cC}$. 
Let $F, G\colon \cC \to \cC$ be autoequivalences of $\cC$. 
Let $a \colon F \to G$ be an object of $\Map_{\Fun_\bC(\cC, \cC)}(F, G)$, 
and let $(a^!)^! \colon F \to G$ be obtained by applying the 
construction of Lemma~\textup{\ref{lemma-adjoint-functoriality}} twice 
\textup{(}note that $F$ and its inverse $F^{-1}$ are mutually left and right adjoint, 
and similarly for $G$ and $G^{-1}$\textup{)}. 
Then there is a commutative diagram of functors
\begin{equation*}
\xymatrix{
\rS_{\cC} \circ F \ar[rr]^{\rS_{\cC} a} \ar[d]_[@!-90]{\sim} && \rS_{\cC} \circ G \ar[d]^[@!-90]{\sim} \\
F \circ \rS_{\cC} \ar[rr]^{(a^!)^! \rS_{\cC}} && G \circ \rS_{\cC}
}
\end{equation*}
where the vertical arrows are equivalences. 
\end{lemma}

\begin{proof}
We have functorially in $C,D \in \cC$ a diagram 
\begin{equation*}
\xymatrix{
\ulMap_{\cC}(C, \rS_{\cC} \circ F(D)) \ar[d]_[@!-90]{\sim} \ar[r] & \ulMap_\cC(C, \rS_{\cC} \circ G(D)) \ar[d]^[@!-90]{\sim} \\ 
\ulMap_{\cC}(F(D), C)^{\vee} \ar[d]_[@!-90]{\sim} \ar[r] & \ulMap_{\cC}(G(D), C)^{\vee} \ar[d]^[@!-90]{\sim} \\
\ulMap_{\cC}(D, F^{-1}(C))^{\vee} \ar[d]_[@!-90]{\sim} \ar[r] & \ulMap_{\cC}(D, G^{-1}(C))^{\vee} \ar[d]^[@!-90]{\sim} \\
\ulMap_{\cC}(F^{-1}(C),  \rS_{\cC} (D)) \ar[d]_[@!-90]{\sim} \ar[r] & \ulMap_{\cC}(G^{-1}(C), \rS_{\cC}(D)) \ar[d]^[@!-90]{\sim}  \\
\ulMap_{\cC}(C, F \circ  \rS_{\cC} (D))  \ar[r] & \ulMap_{\cC}(C, G \circ \rS_{\cC}(D)) .
}
\end{equation*} 
The first and third vertical equivalences in the diagram are given by the defining 
equivalences of a Serre functor, and the second and fourth are given by the 
adjunctions between $F$ and $F^{-1}$ and $G$ and $G^{-1}$. 
The first and second horizontal morphisms are induced by $a\colon F \to G$, 
the third and fourth by $a^!\colon G^{-1} \to F^{-1}$, and the last by 
$(a^!)^!\colon F \to G$. 
The first and third squares in the diagram commute since the 
definining equivalences of a Serre functor are functorial, 
and the second and fourth squares commute by 
Lemma~\ref{lemma-adjoint-functoriality}. Hence the diagram 
is commutative. 
The outer square in the diagram thus induces the desired diagram of functors. 
\end{proof}

\begin{proof}[Proof of Proposition~\textup{\ref{proposition-serre-HC}}]
By the definition of the action of $\rS_{\cC}$ on $\HH^\bullet(\cC)$, it suffices 
to show that for any $a \colon \id_{\cC} \to \id_{\cC}[n]$ representing an element of 
$\HH^n(\cC)$, there is a commutative diagram 
\begin{equation*}
\xymatrix{
\rS_{\cC} \circ \id_{\cC} \ar[rr]^{\rS_{\cC} a} \ar[d]_[@!-90]{\sim} && \rS_{\cC} \circ \id_{\cC}[n] \ar[d]_[@!-90]{\sim} \\
\id_{\cC} \circ \rS_{\cC} \ar[rr]^{a \rS_{\cC}} && \id_{\cC}[n] \circ \rS_{\cC}
}
\end{equation*} 
where the vertical arrows are the canonical identifications. 
This follows from Lemma~\ref{lemma-serre-HC} with 
$F = \id_{\cC}$ and $G = \id_{\cC}[n]$. 
Indeed, in this case it is easy to see that $(a^!)^! = a$. 
\end{proof}

%%%%%%%%%%%%%%%%%%%%%%%%%%%%%%%%%%%%%%%%%%

\section{Application to fractional Calabi--Yau categories} 
\label{section-fCY}

In this section, we deduce Corollary~\ref{corollary-Z2-HC-intro} from the introduction 
(restated as Corollary~\ref{corollary-Z2-HC} below), and apply it to the example 
of quartic fourfolds. 

\subsection{Hochschild cohomology in the $\bZ/2$ fractional Calabi--Yau case}
In the statement of Corollary~\ref{corollary-Z2-HC-intro}, 
the Hochschild homology $\HH_\bullet(\cC)$ of $\cC$ appears. 
In our context, this invariant can be defined in terms of Hochschild 
cohomology with coefficients as follows. 
Hochschild homology was also defined in this spirit in \cite{caldararu-HKRII, kuznetsov2009hochschild}, 
which we refer to for comparisons with other possible definitions of this invariant. 

\begin{definition}
\label{definition-hochschild-homology}
Let $\cC$ be a proper $\bC$-linear category  
admitting a Serre functor $\rS_{\cC}$. 
The \emph{Hochschild homology} of $\cC$ is defined as 
\begin{equation*}
\HH_{\bullet}(\cC) = \HH^{\bullet}(\cC, \rS_{\cC}) = 
\rH^\bullet(\ulMap_{\Fun_{\bC}(\cC, \cC)}(\id_{\cC}, \rS_{\cC})) . 
\end{equation*} 
\end{definition}

\begin{remark}
\label{remark-HH-CY-category}
If $\cC$ is a Calabi--Yau category, say $\rS_{\cC} = [n]$, then it follows immediately 
from our definitions that there is an isomorphism 
$\HH^{\bullet}(\cC) \cong \HH_{\bullet}(\cC)[-n]$ of graded vector spaces. 
\end{remark}

In the geometric case, Hochschild homology can be computed in terms of Hodge cohomology 
by means of the Hochschild--Kostant--Rosenberg theorem; the following formulation 
of this result, which holds by  \cite[Theorem 4.8]{yekutieli} and~\cite[Corollary 1.5]{HKRcharp}, 
is the one we shall use below. 
We write $\HH_\bullet(X) = \HH_\bullet(\Perf(X))$ for a scheme $X$. 

\begin{theorem}
\label{theorem-HKR} 
Let $X$ be a smooth proper scheme over $\bC$. 
Assume that $\characteristic(\bC) = 0$ or that $\characteristic(\bC) \geq \dim(X)$. 
Then there is an isomorphism 
\begin{equation*}
\HH_n(X) \cong \bigoplus_{p-q = n} \rH^q(X, \Omega^p_X). 
\end{equation*}
\end{theorem}

By combining Theorem~\ref{theorem-equivariant-HC} and Proposition~\ref{proposition-serre-HC}, 
we obtain the following. 

\begin{corollary}
\label{corollary-Z2-HC}
Let $\cC$ be a proper $\bC$-linear category such that 
$\rS_{\cC} = \sigma \circ [n]$ is a Serre functor, where $n$ is an integer 
and $\sigma$ is the autoequivalence corresponding to the generator of a 
$\bZ/2$-action on $\cC$. 
Assume the characteristic of $\bC$ is not $2$. 
Then there is an isomorphism 
\begin{equation*}
\HH^\bullet(\cC^{\bZ/2}) \cong 
\HH^{\bullet}(\cC) \oplus (\HH_{\bullet}(\cC)^{\bZ/2}[-n]), 
\end{equation*}
where $\bZ/2$ acts on $\HH_{\bullet}(\cC)$ via conjugation 
by $\sigma$ on $\Fun_{\bC}(\cC, \cC)$.
\end{corollary}

\begin{proof}
Theorem~\ref{theorem-equivariant-HC} gives 
\begin{equation*}
\HH^\bullet(\cC^{\bZ/2}) \cong 
\HH^{\bullet}(\cC)^{\bZ/2} \oplus 
\HH^\bullet(\cC, \sigma)^{\bZ/2} . 
\end{equation*} 
Since $\sigma =  \rS_{\cC} \circ [-n]$ the second summand is as claimed, 
and it suffices to see that $\sigma$ acts trivially on $\HH^\bullet(\cC)$. 
But the shift functor $[-n]$ clearly acts trivially on 
$\HH^{\bullet}(\cC)$, and so does~$\rS_{\cC}$ by 
Proposition~\ref{proposition-serre-HC}. 
\end{proof}

The invariant category $\cC^{\bZ/2}$ appearing in Corollary~\ref{corollary-Z2-HC} 
is Calabi--Yau. More generally, we have the following. 

\begin{lemma}
\label{corollary-fCY-invariants}
Let $\cC$ be a proper $\bC$-linear category such that 
$\rS_{\cC} = \sigma \circ [n]$ is a Serre functor, where $n$ is an integer and 
$\sigma$ is the autoequivalence corresponding to the 
generator of a $\bZ/q$-action on $\cC$. 
Assume $q$ is coprime to the characteristic of $\bC$. 
Then the shift functor $[n] \colon \cC^{\bZ/q} \to \cC^{\bZ/q}$ is a Serre functor 
for $\cC^{\bZ/q}$. 
\end{lemma}

\begin{proof}
Let $G = \bZ/q$. 
Note that $\cC^G$ is proper, by the description of the mapping 
spaces~\eqref{mapping-space-invariants} and our assumption on the characteristic of $\bC$. 
By the universal property of $\cC^G \to \cC$, the Serre functor $\rS_{\cC} = \sigma \circ [n]$ 
induces an autoequivalence $\rS_{\cC}^G \colon \cC^G \to \cC^G$. 
We show below that $\rS_{\cC}^G$ is indeed a Serre functor for $\cC^G$; 
since the autoequivalence of $\cC^G$ induced by $\sigma$ is equivalent to the 
identity, this will prove the lemma. 
More generally, the following argument proves that if $\cC$ has an action by a finite 
group $G$ of order coprime to the characteristic of $\bC$, and $\cC$ admits a Serre 
functor $\rS_{\cC}$ such that the composition 
\begin{equation*}
\cC^G \xrightarrow{p} \cC \xrightarrow{\rS_{\cC}} \cC
\end{equation*} 
lifts to an autoequivalence $\rS_{\cC}^G \colon \cC^G \to \cC^G$, then 
$\rS_{\cC}^G$ is a Serre functor for $\cC^G$. 
The condition that $\rS_{\cC} \circ p$ lifts is presumably automatic (it is closely related to 
Lemma~\ref{lemma-serre-HC}), but we do not address this point here. 

We have functorially in $\wtilde{C}, \wtilde{D} \in \cC^G$ equivalences 
\begin{equation}
\label{eq-serre-invariants-1}
\ulMap_{\cC^G}(\wtilde{C}, \rS_{\cC}^G(\wtilde{D})) \simeq 
\ulMap_{\cC}(C, \rS_{\cC}(D))^G \simeq 
(\ulMap_{\cC}(D, C)^{\vee})^G . 
\end{equation}
By functoriality of the defining equivalences of a Serre functor, the 
action of $G$ on $\ulMap_{\cC}(D, C)^{\vee}$ is induced by the action 
of $G$ on $\ulMap_{\cC}(D, C)$. Moreover, by definition 
\begin{equation*}
\ulMap_{\cC}(D, C)^{\vee} =  \ulMap_{\Vect_{\bC}}(\ulMap_{\cC}(D, C), \bC), 
\end{equation*} 
so since the formation of the mapping space $\ulMap_{\Vect_{\bC}}(-,-)$ 
takes colimits in the first variable to limits, we find 
\begin{equation}
\label{eq-serre-invariants-2}
(\ulMap_{\cC}(D, C)^{\vee})^G \simeq (\ulMap_{\cC}(D, C)_G)^{\vee}. 
\end{equation}
Further, we have 
\begin{equation}
\label{eq-serre-invariants-3}
\ulMap_{\cC}(D, C)_G \simeq \ulMap_{\cC}(D,C)^G 
\simeq \ulMap_{\cC^G}(\wtilde{D}, \wtilde{C}), 
\end{equation}
where the first equivalence is given by the norm map 
(using the assumption on the characteristic of $\bC$). 
Finally, combining~\eqref{eq-serre-invariants-1}--\eqref{eq-serre-invariants-3} 
gives the required functorial equivalences 
\begin{equation*}
\ulMap_{\cC^G}(\wtilde{C}, \rS_{\cC}^G(\wtilde{D})) 
\simeq  \ulMap_{\cC^G}(\wtilde{D}, \wtilde{C})^\vee. \qedhere
\end{equation*} 
\end{proof}

\begin{remark}
\label{remark-CY-cover}
In the situation of Lemma~\ref{corollary-fCY-invariants}, 
the category $\cC^{\bZ/q}$ should be regarded as the ``canonical Calabi--Yau cover'' of $\cC$. 
Indeed, assume $X$ is a smooth proper $\bC$-scheme of dimension~$n$, 
whose canonical bundle satisfies $\omega_X^q \cong \cO_X$. 
Then $\cC = \Perf(X)$ satisfies the assumptions of 
Lemma~\ref{corollary-fCY-invariants} with $\bZ/q$-action given by tensoring with $\omega_X$, 
and the invariant category $\cC^{\bZ/q}$ recovers the derived category of the 
canonical Calabi--Yau cover of $X$. 
Namely, there is a $q$-fold \'{e}tale cover $Y = \Spec_X(\cR) \to X$, where 
\begin{equation*}
\cR = \cO_X \oplus \omega_X \oplus \dots \oplus \omega_X^{q-1}
\end{equation*}
with algebra structure determined by $\omega_X^q \cong \cO_X$.
The variety $Y$ is Calabi--Yau in the sense that $\omega_{Y} \cong \cO_{Y}$, 
and there is an equivalence $\cC^{\bZ/q} \simeq \Perf(Y)$ (see~\cite[Theorem~1.2]{elagin-equivariant}). 
\end{remark}

\begin{example}
Assume the characteristic of $\bC$ is not $2$, and let $X$ be an Enriques surface. 
Since $\omega_X^2 \cong \cO_X$ we are in the situation of Remark~\ref{remark-CY-cover}; the 
double cover $Y \to X$ is the K3 surface associated to $X$. 
Using the HKR isomorphisms (Theorem~\ref{theorem-HKR} and~\eqref{HKR-HC}), we find 
\begin{align*}
\HH^\bullet(Y) & \cong \bC[0] \oplus \bC^{22}[-2] \oplus \bC[-4] ,  \\ 
\HH^\bullet(X) & \cong \bC[0] \oplus \bC^{10}[-2] \oplus \bC[-4] , \\ 
\HH_\bullet(X) & \cong  \bC^{12}[0]  .  
\end{align*} 
This is of course consistent with the conclusion 
\begin{equation*}
\HH^\bullet(Y) \cong 
\HH^{\bullet}(X) \oplus (\HH_{\bullet}(X)^{\bZ/2}[-2]) 
\end{equation*}
of Corollary \ref{corollary-Z2-HC}. 
Note that this also implies that the $\bZ/2$-action on $\HH_\bullet(X)$ is trivial. 
\end{example} 

\subsection{Quartic fourfolds} 
We assume for this subsection that the characteristic of $\bC$ is 
not $2$ or $3$; 
we state explicitly where this assumption is used below. 
Let $X \subset \bP^5$ be a smooth quartic fourfold over $\bC$, and 
let $Y \to \bP^5$ be the double cover of $\bP^5$ branched along 
$X$. Let $\cC_X \subset \Perf(X)$ and $\cC_Y \subset \Perf(Y)$ 
be the $\bC$-linear categories defined by the 
semiorthogonal decompositions 
\begin{align}
\label{equation-CX} \Perf(X) & = \langle \cC_X, \cO_X, \cO_X(1) \rangle , \\ 
\label{equation-CY} \Perf(Y) & = \langle \cC_Y, \cO_Y, \cO_Y(1), \cO_Y(2), \cO_Y(3) \rangle . 
\end{align}
The fact that there are semiorthogonal decompositions of these forms follows 
by an easy coherent cohomology computation. 

The following result guarantees the conditions of 
Corollary~\ref{corollary-Z2-HC} hold for $\cC_X$, and 
identifies the category of $\bZ/2$-invariants with $\cC_Y$. 

\begin{lemma}
\label{lemma-CX-quartic}
There is an autoequivalence $\sigma$ of $\cC_X$ 
corresponding to the generator of a $\bZ/2$-action, such that:
\begin{enumerate}
\item \label{serre-functor-quartic}
$\rS_{\cC_X} = \sigma \circ [3]$ is a Serre functor for $\cC_X$. 
\item \label{invariants-quartic}
There is an equivalence $\cC_X^{\bZ/2} \simeq \cC_Y$. 
\end{enumerate}
Further, $\rS_{\cC_Y} = [3]$ is a Serre functor for $\cC_Y$. 
\end{lemma}

\begin{proof}
Part~\eqref{serre-functor-quartic} is a special case 
of~\cite[Corollary~4.1]{kuznetsov2015calabi}. 
More precisely, the proof of~\cite[Theorem~3.5]{kuznetsov2015calabi} 
shows we can take $\sigma = \sO^2_X[-1]$, where 
$\sO_X \colon \cC_X \to \cC_X$ is the ``rotation functor''. 
Now part~\eqref{invariants-quartic} follows 
from~\cite[Proposition~7.10]{cyclic-covers} (we use $\characteristic(\bC) \neq 2$ here).  
Strictly speaking, these results are stated in the references 
at the level of homotopy categories of $\cC_X$ and $\cC_Y$, 
but they also hold at the level of $\bC$-linear categories. 
Finally, the statement about the Serre functor of $\cC_Y$ 
follows from Lemma~\ref{corollary-fCY-invariants}; alternatively, 
it holds by~\cite[Corollary~4.6]{kuznetsov2015calabi}. 
\end{proof}

To analyze $\HH^\bullet(\cC_X)$ using Corollary~\ref{corollary-Z2-HC}, we 
will need the following computations of $\HH_\bullet(\cC_X)$ and $\HH^{\bullet}(\cC_Y)$. 

\begin{lemma}
\label{lemma-HH-CX-CY}
There are isomorphisms of graded vector spaces 
\begin{align}
\label{HH-CX} \HH_{\bullet}(\cC_X) & \cong \bC^{21}[2] \oplus \bC^{144} \oplus \bC^{21}[-2] , \\
\label{HC-CY} \HH^{\bullet}(\cC_Y) & \cong \bC \oplus \bC^{90}[-2] \oplus \bC^2[-3] \oplus \bC^{90}[-4] \oplus \bC[-6]. 
\end{align}
\end{lemma} 

\begin{proof}
It is straightforward to check that the Hodge diamonds of $X$ and $Y$ 
are, respectively, as follows:  
\begin{equation*}
\begin{smallmatrix}
&&&& 1 \\
&&& 0 && 0 \\
&& 0 && 1 && 0 \\
& 0 && 0 && 0 && 0 \\
0 && 21 && 142 && 21 && 0 \\ 
& 0 && 0 && 0 && 0 \\
&& 0 && 1 && 0 \\
&&& 0 && 0 \\
&&&& 1 
\end{smallmatrix}
\qquad \text{and} \qquad 
\begin{smallmatrix}
&&&&& 1 \\
&&&& 0 && 0 \\
&&& 0 && 1 && 0 \\
&& 0 && 0 && 0 && 0 \\
& 0 && 0 && 1 && 0 && 0 \\
0 && 1 && 90 && 90 && 1 && 0 \\
& 0 && 0 && 1 && 0 && 0 \\
&& 0 && 0 && 0 && 0 \\
&&& 0 && 1 && 0 \\
&&&& 0 && 0 \\
&&&&& 1 
\end{smallmatrix}
\end{equation*}
By HKR (Theorem~\ref{theorem-HKR}) and our assumption on the characteristic of $\bC$, we obtain 
\begin{align*}
\HH_\bullet(X) & \cong \bC^{21}[2] \oplus \bC^{146} \oplus \bC^{21}[-2] , \\
\HH_{\bullet}(Y) & \cong \bC[3] \oplus \bC^{90}[1] \oplus \bC^6 \oplus \bC^{90}[-1] \oplus \bC[-3]. 
\end{align*} 
By the additivity of Hochschild homology under semiorthogonal decompositions (see \cite{kuznetsov2009hochschild}), 
the defining semiorthogonal decompositions~\eqref{equation-CX} and~\eqref{equation-CY}, and 
the isomorphism of graded vector spaces $\HH^\bullet(\cC_Y) \cong \HH_{\bullet}(\cC_Y)[-3]$ 
(see Remark~\ref{remark-HH-CY-category}), 
we obtain the desired isomorphisms~\eqref{HH-CX} and~\eqref{HC-CY}. 
\end{proof}

Putting the above together with Corollary~\ref{corollary-Z2-HC}, we obtain: 

\begin{lemma}
\label{lemma-HC-CX}
The canonical map $\HH^\bullet(\cC_X) \to \HH^\bullet(\cC_Y)$ given by 
Corollary~\textup{\ref{corollary-Z2-HC}} is an isomorphism in degrees not equal to $3$. 
Explicitly, we have 
\begin{equation*}
\HH^{\bullet}(\cC_X)  \cong \bC \oplus \bC^{90}[-2] \oplus \bC^d[-3] \oplus \bC^{90}[-4] \oplus \bC[-6] 
\end{equation*}
for some $0 \leq d \leq 2$. 
\end{lemma}

\begin{proof}
By Corollary~\ref{corollary-Z2-HC} combined with Lemma~\ref{lemma-CX-quartic}, 
there is a splitting 
\begin{equation*}
\HH^\bullet(\cC_Y) \cong \HH^\bullet(\cC_X) \oplus (\HH_\bullet(\cC_X)^{\bZ/2}[-3]). 
\end{equation*}
Now the result follows for degree reasons from the computation of Lemma~\ref{lemma-HH-CX-CY}. 
\end{proof}

\begin{remark}
The value of $d$ in Lemma~\ref{lemma-HC-CX} would follow from a better understanding 
of the action of $\bZ/2$ on $\HH_0(\cC_X)$. 
Namely, the above proof shows that we have $\HH_{\bullet}(\cC_X)^{\bZ/2} \cong \bC^{2-d}$. 
In particular, the action of $\bZ/2$ on $\HH_\bullet(\cC_X)$ is certainly nontrivial. 
\end{remark}

Finally, we explain how Lemma~\ref{lemma-HC-CX} implies Proposition~\ref{proposition-quartic-deformation} 
from the introduction, using the results of~\cite{kuznetsov2015height}. 
First, we note that there is a canonical restriction morphism 
\begin{equation*}
\HH^\bullet(X) \to \HH^\bullet(\cC_X), 
\end{equation*}
see~\cite[\S3.1]{kuznetsov2015height}. 
By HKR~\eqref{HKR-HC} there is a decomposition   
\begin{equation*}
\HH^2(X) \cong \rH^0(X, \wedge^2 \rT_X) \oplus \rH^1(X, \rT_X) \oplus \rH^2(X, \cO_X). 
\end{equation*}
We aim to show that the composition 
\begin{equation*}
\rH^1(X, \rT_X) \to \HH^2(X) \to \HH^2(\cC_X) 
\end{equation*}
is an isomorphism. 
In fact, we show slightly more: 

\begin{proposition}
Both maps $\rH^1(X, \rT_X) \to \HH^2(X)$ and 
$\HH^2(X) \to \HH^2(\cC_X)$ are isomorphisms. 
\end{proposition}

\begin{proof}
We have 
\begin{equation*}
\dim \rH^1(X, \rT_X) = 90 = \dim \HH^2(\cC_X). 
\end{equation*}
The first equality holds by an easy computation, and the second by Lemma~\ref{lemma-HC-CX}. 
Hence the claim amounts to the injectivity of $\HH^2(X) \to \HH^2(\cC_X)$. 
An easy computation shows that the pseudoheight of the exceptional collection 
$\cO_X, \cO_X(1) \in \Perf(X)$, as defined in \cite[Definition~4.4]{kuznetsov2015height}, 
is equal to $3$. 
Thus \cite[Corollary~4.6]{kuznetsov2015height} gives the injectivity of 
$\HH^2(X) \to \HH^2(\cC_X)$. 
\end{proof}

%%%%%%%%%%%%%%%%%%%%%%%%%%%%%%%%%%%%

\providecommand{\bysame}{\leavevmode\hbox to3em{\hrulefill}\thinspace}
\providecommand{\MR}{\relax\ifhmode\unskip\space\fi MR }
\providecommand{\MRhref}[2]{%
  \href{http://www.ams.org/mathscinet-getitem?mr=#1}{#2}
}
\providecommand{\href}[2]{#2}

%%%%%%%%%%%%%%%%%%%%%%%%%%%%%%%%%%%%


\begin{thebibliography}{10}

\bibitem{addington-thomas}
Nicolas Addington and Richard Thomas, \emph{{Hodge theory and derived
  categories of cubic fourfolds}}, Duke Math. J. \textbf{163} (2014), no.~10,
  1885--1927.

\bibitem{HKRcharp}
Benjamin Antieau and Gabriele Vezzosi, \emph{A remark on the
  {H}ochschild--{K}ostant--{R}osenberg theorem in characteristic $p$},
  arXiv:1710.06039 (2017).

\bibitem{orbiHKR-caldararu}
Dima Arinkin, Andrei C{\u{a}}ld{\u{a}}raru, and M\'{a}rton Hablicsek,
  \emph{Formality of derived intersections and the orbifold {HKR} isomorphism},
  to appear in J. Algebra, arXiv:1412.5233 (2014).

\bibitem{families-stability}
Arend Bayer, Mart\'i Lahoz, Emanuele Macr\`\i, Howard Nuer, Alexander Perry,
  and Paolo Stellari, \emph{Families of stability conditions}, {in preparation}
  (2018).

\bibitem{BLMS}
Arend Bayer, Mart\'i Lahoz, Emanuele Macr\`\i, and Paolo Stellari,
  \emph{Stability conditions on {K}uznetsov components}, arXiv:1703.10839
  (2017).

\bibitem{bzfn}
David Ben-Zvi, John Francis, and David Nadler, \emph{Integral transforms and
  {D}rinfeld centers in derived algebraic geometry}, J. Amer. Math. Soc.
  \textbf{23} (2010), no.~4, 909--966.

\bibitem{bondal-kapranov}
Alexei Bondal and Mikhail Kapranov, \emph{{Representable functors, Serre
  functors, and mutations}}, Mathematics of the USSR-Izvestiya \textbf{35}
  (1990), no.~3, 519.

\bibitem{caldararu-HKRII}
Andrei C{\u{a}}ld{\u{a}}raru, \emph{The {M}ukai pairing. {II}. {T}he
  {H}ochschild-{K}ostant-{R}osenberg isomorphism}, Adv. Math. \textbf{194}
  (2005), no.~1, 34--66.

\bibitem{DG-vs-stable}
Lee Cohn, \emph{Differential graded categories are k-linear stable infinity
  categories}, arXiv:1308.2587 (2016).

\bibitem{edixhoven}
Bas Edixhoven, \emph{N\'eron models and tame ramification}, Compositio Math.
  \textbf{81} (1992), no.~3, 291--306.

\bibitem{elagin-equivariant}
Alexey Elagin, \emph{On equivariant triangulated categories}, arXiv:1403.7027
  (2015).

\bibitem{gaitsgory-DAG}
Dennis Gaitsgory and Nick Rozenblyum, \emph{A study in derived algebraic
  geometry}, available at \url{http://www.math.harvard.edu/~gaitsgde/GL/}.

\bibitem{ganter}
Nora Ganter, \emph{Inner products of 2-representations}, Adv. Math.
  \textbf{285} (2015), 301--351.

\bibitem{huy-ren}
Daniel Huybrechts and J{\o}rgen Rennemo, \emph{Hochschild cohomology versus the
  {J}acobian ring, and the {T}orelli theorem for cubic fourfolds}, to appear in
  Algebr. Geom., arXiv:1610.04128 (2016).

\bibitem{kuznetsov2009hochschild}
Alexander Kuznetsov, \emph{{Hochschild homology and semiorthogonal
  decompositions}}, arXiv:0904.4330 (2009).

\bibitem{kuznetsov2010derived}
\bysame, \emph{{Derived categories of cubic fourfolds}}, {Cohomological and
  geometric approaches to rationality problems}, Springer, 2010, pp.~219--243.

\bibitem{kuznetsov2015height}
\bysame, \emph{Height of exceptional collections and {H}ochschild cohomology of
  quasiphantom categories}, J. Reine Angew. Math. \textbf{708} (2015),
  213--243.

\bibitem{kuznetsov2015calabi}
\bysame, \emph{{Calabi--Yau and fractional Calabi--Yau categories}}, to appear
  in J. Reine Angew. Math., arXiv:1509.07657 (2016).

\bibitem{cyclic-covers}
Alexander Kuznetsov and Alexander Perry, \emph{Derived categories of cyclic
  covers and their branch divisors}, Selecta Math. (N.S.) \textbf{23} (2017),
  no.~1, 389--423.

\bibitem{GM-derived-categories}
\bysame, \emph{Derived categories of {G}ushel--{M}ukai varieties}, to appear in
  Compos. Math. arXiv:1605.06568 (2017).

\bibitem{lurie-HA}
Jacob Lurie, \emph{Higher algebra}, available at
  \url{http://www.math.harvard.edu/~lurie/}.

\bibitem{lurie-SAG}
\bysame, \emph{Spectral algebraic geometry}, available at
  \url{http://www.math.harvard.edu/~lurie/}.

\bibitem{lurie-HTT}
\bysame, \emph{Higher topos theory}, Annals of Mathematics Studies, vol. 170,
  Princeton University Press, Princeton, NJ, 2009.

\bibitem{galois-akhil}
Akhil Mathew, \emph{The {G}alois group of a stable homotopy theory}, Adv. Math.
  \textbf{291} (2016), 403--541.

\bibitem{NCHPD}
Alexander Perry, \emph{Noncommutative homological projective duality},
  arXiv:1804.00132 (2018).

\bibitem{sosna}
Pawel Sosna, \emph{Linearisations of triangulated categories with respect to
  finite group actions}, Math. Res. Lett. \textbf{19} (2012), no.~5,
  1007--1020.

\bibitem{yekutieli}
Amnon Yekutieli, \emph{The continuous {H}ochschild cochain complex of a
  scheme}, Canad. J. Math. \textbf{54} (2002), no.~6, 1319--1337.

\end{thebibliography}
\end{document}